\documentclass[a4paper,12pt]{article}
\usepackage{graphicx,amssymb}
\usepackage{amssymb,amsmath}
\usepackage{hyperref}
\usepackage{amsfonts}
\usepackage{amscd}
\usepackage{amsthm}
\usepackage{color}
\usepackage[utf8]{inputenc}
\usepackage{longtable}
\numberwithin{equation}{section}

\usepackage{amssymb,amsmath,amsthm,authblk}
\usepackage{enumitem}
\newtheorem{theorem}{Theorem}
\newtheorem{corollary}{Corollary}
\newtheorem{remark}{Remark}
\newtheorem{lemma}{Lemma}
\newtheorem{proposition}{Proposition}
\newtheorem{definition}{Definition}
\newtheorem{observation}{Observation}

\begin{document}
\title{ \textbf{Group Vertex Magicness of $H$-join and Generalised Friendship Graph}}
\baselineskip 16pt
\author{\large S. Balamoorthy, S.V. Bharanedhar \footnote{Corresponding Author. \\ E-Mail addresses: moorthybala545@gmail.com (S. Balamoorthy), bharanedhar3@gmail.com (S.V. Bharanedhar)  }\\
{\small Department of Mathematics,\\ Central University of Tamil Nadu,\\ Thiruvarur 610005,\\ India.}}
\date{\today}
\maketitle

\begin{abstract}
Let $A$ be a non-trivial Abelian group. A graph $G$ is $A$-vertex magic if there exists a labeling $l : V(G) \to A\setminus\{0\}$ and $\mu$ in $A$ such that $w(v)= \sum_{u \in N_G(v)} l(u) = \mu $ for any vertex $v$ of $G$, where $N_G(v)= \{u\in V(G) : vu \in E(G)\}$ is the open neighborhood of the vertex $v$ in $G$. In this paper, we characterize $\mathbb{Z}$-vertex magic graphs. We prove that $G$ is $A$-vertex magic if its reduced graph  is $A$-vertex magic, where $A$ is a non-trivial Abelian group. We also construct a family of $A$-vertex magic graphs using $H$-generalized join operation. In addition, we  introduce a new type of labeling called $A'$-vertex magic labeling of graphs and characterize $A$-vertex  magicness using $A'$-vertex magicness. We give a new procedure to embed any graph as an induced subgraph of an $A$-vertex magic graph and we construct infinite families of $A$-vertex magic graphs both of these procedure uses less number of vertices compared to the one given in (Sabeel et al. in Australas. J. Combin. 85(1) (2023), 49-60). Moreover, we generalize some results proved in (Sabeel et al. in Australas. J. Combin. 85(1) (2023), 49-60). Finally we completely classify  generalised friendship graph using $A$-vertex magicness and group vertex magicness. \medskip
\end{abstract}
\section{Introduction}
All the graphs considered in this paper are connected simple  finite graphs and $A$ denote a non-trivial Abelian group, not necessarily finite. Let $G=(V(G),E(G))$ be a graph, we denote by $V(G)$  the vertex set of $G$.
A graph $G$ is said to be $A$-vertex magic if there exists a labeling $l : V(G) \to A\setminus\{0\}$ and $\mu$ in $A$ such that $w(v)= \sum_{u \in N_G(v)} l(u) = \mu $ for any vertex $v$ of $G$, where $N_G(v)= \{u\in V(G) : vu \in E(G)\}$ is the open neighborhood of the vertex $v$ in $G$.  The element $\mu$ is called a magic constant of the labeling $l$. A graph $G$ that admits such a labeling is called an $A$-vertex magic graph. If $G$ is an $A$-vertex magic graph for every non-trivial Abelian group $A$, then $G$ is called a group vertex magic graph. We denote the degree of a vertex $v$ in a graph $G$ by $d_G(v)$. If $v\in V(G)$ is such that $d_G(v)=1$, then $v$ is called a pendant vertex of $G$. In this case, the vertex to which $v$ is adjacent is called a support vertex. A vertex $w$ is called a strong support vertex if $w$ is adjacent to at least two pendant vertices.

The concept of group vertex magic graphs was introduced by Kamatchi et al. \cite{kama} and they have studied some properties of this magicness for trees with diameter at most 4.
The main motivation of group vertex magic graph is group magic graph, which was introduced by Lee et al. \cite{lee0} and this work was further stuided by Lee et al. \cite{lee1}, Low and Lee  \cite{lee2}. In \cite{Sabeel}, Sabeel et al. have completely characterized $A$- vertex magic trees of diameter at most $5$, where $A$ is a finite Abelian group. 
In \cite{Bala1}, the present authors have given a necessary condition for a graph to be $A$-vertex magic and they have obtained results for the $A$-vertex magicness of product of graphs.

In this paper, we use group elements to label the vertices of a graph and we have extended the study of group vertex magicness of a graph by considering an arbitrary Abelian group.
This work considers the vertex magicness of $H$-join and generalised friendship graphs. We use the following important results in group theory namely Cauchy’s theorem, Sylow’s first theorem and fundamental theorem of finite Abelian groups. We refer to Bondy and Murty \cite{bondy1976graph} for graph theoretic terminology and notations. We refer to Herstein \cite{herstein} for ideas in algebra.

Let $H$ be a graph with vertex set $\{v_1,v_2,\ldots,v_k\}$ and let $G_1,G_2,\ldots,G_k$ be arbitrary graphs. The $H$-join operation of the graphs $G_1,G_2,\ldots,G_k$ is denoted as $H[G_1,G_2,\ldots,G_k]$,  is obtained by replacing the vertex $v_i$ of $H$ by the graph $G_i$ for $1 \leq i \leq k$ and every vertex of $G_i$ is made adjacent with every vertex of $G_j$, whenever $v_i$ is adjacent to $v_j$ in $H$. The vertex set $V(G)= \displaystyle \bigcup_{i=1}^k V(G_i)$ and edge set $E(G) =  \displaystyle \big(\bigcup_{i=1}^k E(G_i) \big) \cup \big(\bigcup_{v_iv_j \in E(H)} \{uv : u \in V(G_i), v \in V(G_j)\}\big)$ (see \cite{Saravanan} and \cite{Vinoth}). If $G\cong G_i$, for $1\leq i\leq k$, then $H[G_1,G_2,\ldots, G_k]\cong H[G]$, the \emph{lexicographic product} of $H$ and $G$. 

Let $R$ be a commutative ring an element $x$ of $R$ is said to be a unit if it has a multiplicative inverse. A non-zero element $u$ of R is said to be a zero-divisor if its product with a non-zero element yields zero. Let $Z(R)$ denotes set of all zero-divisor of $R$. Given a ring $R$, we associate  a simple graph $\Gamma(R)$ to $R$ with vertex set the set of all zero-divisor $Z(R)$ of $R$ and two distinct vertices $u$ and $v$ are adjacent if and only if $uv=0$. This graph $\Gamma(R)$ is called the zero-divisor graph of $R$.

A caterpillar is a tree in which the removal of pendant vertices yields a path.

\begin{observation}\label{BP2007}
If $G$ is a connected graph with $|V(G)| \leq 3$, then $G$ is $A$-vertex magic, where $|A|>2$.
\end{observation}
The graph $P_4$ is not $A$-vertex magic graph with minimum number of edges.

\begin{theorem}\label{BP2001}
Let $G$ be a $\mathbb{Z}$-vertex magic graph. Then there exists  $k\in \mathbb{N}$ such that $G$ is $\mathbb{Z}_m$-vertex magic, for all $m \ge k$. 
\end{theorem}
\begin{proof}
Assume that $l$ is $\mathbb{Z}$-vertex magic labeling of $G$. Then $l(v_i)=k_i$, where $k_i \in \mathbb{Z}, i=1,2,\ldots,n$. Let $k= max\{|k_i|\}+1$. Now, define $l' : V(G) \to \mathbb{Z}_m \setminus \{0\}$ by $l'(v_i)=l(v_i)a$, where $o(a)=m$ and $m \ge k$. Clearly, the labeling $l'$ is $\mathbb{Z}_m$-vertex labeling of $G$, for all $m \ge k$.
\end{proof}

\begin{theorem}\label{BP2002}
Let $m \in \mathbb{N}$ and $l$ is a $\mathbb{Z}_p$-vertex magic labeling of $G$, for all primes $p \ge m$. Then $G$ is $\mathbb{Z}$-vertex magic.
\end{theorem}
\begin{proof}
Let us assume $l$ is $\mathbb{Z}_p$-vertex magic labeling of $G$, for all prime $p \ge m$. Clearly, $l$ is $\mathbb{Z}_{p'}$-vertex magic, where $p'= min\{p : p \ge m\}$. Then $l(v_i)=k_ia,$ where $k_i \in \mathbb{Z}$ and $o(a)=p'$. Now, define $l' : V(G) \to \mathbb{Z}\setminus \{0\}$ by $l'(v_i) =k_ia$, here $a$ is generator of $\mathbb{Z}$. 
Suppose $l'$ is not a $\mathbb{Z}$-vertex magic labeling of $G$, then there exists $v_i,v_j \in V(G)$ such that $w(v_i) \neq w(v_j)$, which implies $w(v_i) \neq w(v_j)$ is also true for all $\mathbb{Z}_{p_1}$, where $p_1$ is a prime number, such that ${p_1}>$ max$\{deg(v_i) : v_i \in V(G)\}\cdot p'$, which is a contradiction. 
\end{proof} 

Let $G$ be a simple graph. In \cite{Anderson} and \cite{Vinoth}, the authors define the following relation on $V(G)$. For any two vertices $u,v \in V(G)$, define $u\sim_G v$ if and only if $N_G(u) = N_G(v)$. It is easy to see that, the relation $\sim_G$ is an equivalence relation on $V(G)$. Let $[u]$ be the equivalence class which contains $u$ and $S$ be the set of all equivalence classes of this relation $\sim_G$. The reduced graph $H$ of $G$ is defined as follows. The \emph{reduced graph} $H$ of $G$ is the graph with vertex set $V(H) = S$ and  two distinct vertices $[u]$ and $[v]$ are adjacent in $H$ if and only if $u$ and $v$ are adjacent in $G$. Note that, if $V(H) = \{[u_1],[u_2],\ldots,[u_k]\}$, then $G$ is the $H$-join of $\langle[u_1]\rangle, \langle[u_2]\rangle,\ldots,\langle[u_k]\rangle$, that is, $G \cong H[\langle[u_1]\rangle, \langle[u_2]\rangle,\ldots,\langle[u_k]\rangle]$  ($\langle [u] \rangle$ denote the subgraph induced by $[u]$)
and each $[u_i]$ is an independent subset 
of $G$, we take $|[u_i]|=m_i$, where $m_i \in \mathbb{N}$ for each $i$. Clearly, $H$ is isomorphic to a subgraph of $G$ induced by $\{u_1,u_2,\ldots, u_k\}$. 

\begin{theorem}\label{BP2003}
If the reduced graph $H$ of $G$ is $A$-vertex magic, where $|A| > 2$ then $G$ is $A$-vertex magic. 
\end{theorem}
\begin{proof}
Assume that $H$ is $A$-vertex magic with the corresponding labeling $l$ such that $l([u_i])= a_i$, where $a_i \in A \setminus \{0\}$ and $[u_i]=\{u_{1}^i,\ldots,u_{m_i}^i\}$, for $i = 1,\ldots,k$.
Define the label on $V(G)$ as follows, suppose $m_i$ is odd, then define \[l'(u_{j}^i)= \begin{cases} \text{$a_i$} &\quad\text{if $j = 1$}\\ \text{$a$} &\quad\text{if $j>1$ and $j$ is odd}\\ \text{$-a$} &\quad\text{if $j$ is even,}\\ \end{cases} \] for some $a \in A\setminus \{0\}$. Suppose $m_i$ is even, then define \[l'(u_{j}^i)= \begin{cases} \text{$a_i+a$} &\quad\text{if $j = 1$}\\  \text{$a$} &\quad\text{if $j>1$ and $j$ is odd}\\ \text{$-a$} &\quad\text{if $j$ is even,}\\ \end{cases} \]where $a_i \neq -a$. Let $v \in V(G)$. Then $v \in [u_i]$, for some $i$. As $w(v)$(corresponding label $l'$ in $G$) is equal to $w([u_i])$(corresponding label $l$ in $H$), $G$ is $A$-vertex magic. 
\end{proof}
The following result proved in \cite{Bala1} is a consequence of Theorem \ref{BP2003}.

\begin{corollary}\label{BP2013}(Theorem 3, \cite{Bala1} )
Let $G = K_{{n_1},{n_2},\ldots,{n_m}}$ be a complete $m$-partite graph. Then $G$ is $A$-vertex magic, where $|A|>2$.
\end{corollary}
\begin{proof}
By the relation $\sim_G$ on $V(G)$, we see that $K_m$ is the reduced graph  of $G$. Since $K_m$ is $A$-vertex magic, where $|A|>2$ and hence the result follows.
\end{proof}

The following  result proved in  \cite{Sabeel} is an immediate consequence of Corollary \ref{BP2013}.

\begin{corollary}(Corollary 2.4, \cite{Sabeel}) For any finite Abelian group $A$ with $|A| \geq 3$, all tree of diameter $2$ are $A$-vertex magic.  
\end{corollary}

The following corollaries are an immediate consequence of Theorem \ref{BP2003}.
\begin{corollary}
If any connected graph $G$ has at most three equivalence classes under the relation $\sim_G$, then the graph $G$ is $A$-vertex magic, where $|A|>2$.
\end{corollary}
\begin{proof}
By  Theorem \ref{BP2003}  and observation \ref{BP2007}, we get the required result.
\end{proof}

\begin{corollary}
Let $G =K_{{n_1},{n_2},\ldots,{n_m}}$ be a complete $m$-partite graph. The $G+K_n$ is $A$-vertex magic.
\end{corollary}
\begin{proof}
The reduced graph $H$ of $G+K_n$ is regular graph. By Theorem \ref{BP2003}, the result follows.
\end{proof}

\begin{theorem}\label{BP2008}
Let $H$ be a simple graph on $k$ vertices and let $G_j ={K_{n_j}}^c$, for all $j= 1,2,\ldots,k$. The graph $G = H[G_1,G_2,\ldots,G_k]$ is $A$-vertex magic, where $|A|>2$, if $n_j >1$, for each $j=1,2,\ldots,k$.
\end{theorem}
\begin{proof}
Let $V(G)= \bigcup_{j=1}^k V(K_{n_j}^c)$ and $v^j_i \in V(K_{n_j}^c)$, where $i = 1, 2,\ldots,n_j$. Assume that each $n_j >1$. Suppose $n_j$ is odd, define $l : V(G) \to A\setminus \{0\}$ by \[ l(v_i^j)=\begin{cases} \text{$a+b$} &\quad\text{if $i=1$}\\ \text{$-b$} &\quad\text{if $i=2$}\\ \text{$-a$} &\quad\text{if $i>1$ and $i$ is odd} \\ \text{$a$} &\quad\text{if $i>2$ and $i$ is even,}\\ \end{cases} \] where $a \neq -b$.\\ Suppose $n_j$ is even, define $l : V(G) \to A\setminus \{0\}$ by \[ l(v_i^j)=\begin{cases} \text{$a$} &\quad\text{if $i$ is odd}\\ \text{$-a$} &\quad\text{if $i$ is even.}\\ \end{cases} \] Thus $w(v) =0,$ for all $v \in \displaystyle V(G)$.
\end{proof} 

The following result proved in \cite{Bala1} is an immediate consequence of Theorem \ref{BP2008}.

\begin{corollary}(Theorem 15, \cite{Bala1}) Let $G$ be any simple graph on $n$ vertices. The graph $G[{K_m}^c]$ is $A$-vertex magic, $|A|>2$ if and only if $m>1$.

\end{corollary}

Consider the zero-divisor graph $\Gamma(R)$, where $R$ is a finite commutative reduced
ring with unity. It is well-known that such a ring has to be a product of finite fields \cite{Vinoth}. That is, there exist $k$ finite fields $F_{q_1},F_{q_2},\ldots,F_{q_k}$ such that $R \cong F_{q_1}\times F_{q_2}\times \ldots \times F_{q_k}$, where ${q_i}'s$ are prime powers, say $q_i = p_i^{n_i}$, $n_i \in \mathbb{N}$. 

\begin{corollary}
Let $q_i\geq 3$ be a prime power, for $1 \leq i \leq k$ and $R \cong F_{q_1}\times F_{q_2}\times \ldots \times F_{q_k}$. Then the graph $\Gamma(R)$ is $A$-vertex magic, where $|A| >2$.
\end{corollary}
\begin{proof}
An easy observation, if $v \in V(\Gamma(R))$, then there exists a vertex $u\in V(\Gamma(R))$ such that $u \neq v$ and $N(u)=N(v)$. Therefore, each equivalence class ($\sim_G$) contains at least two element in $V(\Gamma(R))$. By Theorem $\ref{BP2008}$, $R$ is $A$-vertex magic, where $|A| >2$.
\end{proof}

\begin{definition}
Let $G$ be a graph and let $H$ be its reduced graph. We say that $H$ is $A'$ magic if there exists a labeling $l:V(H)\to A$ such that $0$ can also be used to label the vertices $[u]\in V(H)$
whenever $|[u]|\geq 2$ and $\displaystyle \sum_{[u]\in N([v])} l([u]) = \mu$, for all $[v] \in V(H)$.
\end{definition}

\begin{lemma}\label{BP2011}
Let $A$ be an Abelian group with at least three elements. If $n \geq 2$ and $a \in A $, then there exists $a_1,a_2,\ldots,a_n$ in $A\setminus \{0\}$ such that $a = \sum_{i=1}^n a_i$. \end{lemma}
\begin{proof}
Consider the following cases. 
\item[\bf Case 1.]$a=0$\\
Let $b,c \in A\setminus \{0\}$ and $b \neq -c$.
If $n$ is odd, define  \[a_i = \begin{cases} \text{$b+c$} &\quad\text{if $i=1$}\\ \text{$-c$} &\quad\text{if $i=2$}\\ \text{$-b$} &\quad\text{if $i>1$ and $i$ is odd}\\ \text{$b$} &\quad\text{if $i>2$ and $i$ is even.}\\ \end{cases} \] 
If $n$ is even, define \[a_i = \begin{cases} \text{$b$} &\quad\text{if $i$ is odd}\\ \text{$-b$} &\quad\text{if $i$ is even.}\\ \end{cases} \]
\item[\bf Case 2.]$a \neq 0$\\
Let $a,b \in A\setminus \{0\}$ and $a \neq -b$.
If $n$ is odd, define \[a_i = \begin{cases} \text{$a$} &\quad\text{if $i$ is odd}\\ \text{$-a$} &\quad\text{if $i$ is even.}\\ \end{cases} \]
If $n$ is even, define  \[a_i = \begin{cases} \text{$a+b$} &\quad\text{if $i=1$}\\ \text{$-b$} &\quad\text{if $i=2$}\\ \text{$-a$} &\quad\text{if $i>1$ and $i$ is odd}\\ \text{$a$} &\quad\text{if $i>2$ and $i$ is even.}\\ \end{cases} \] 
\end{proof}

The following result proved in $\cite{Sabeel}$ is an immediate consequence of Lemma $\ref{BP2011}$.

\begin{corollary}(Lemma 2.1 in \cite{Sabeel})
Let $A$ be a finite Abelian group with $|A|\geq 3$ and let $g \in A$. Then, for each $n \geq 2$. there exist $a_1,a_2,\ldots,a_n$ in $A \setminus \{0\}$ such that $g = a_1+a_2+ \cdots +a_n $.  
\end{corollary}

\begin{theorem} \label{BP2009}
A graph $G$ is $A$-vertex magic, where $|A|>2$ if and only if it's reduced graph $H$ of $G$ is $A'$ magic.
\end{theorem}
\begin{proof}
Assume that $G$ is $A$-vertex magic and $l$ is corresponding  $A$-vertex magic labeling. Define ${l}' : V(H) \to A$ by ${l}'([v])= \displaystyle \sum_{\stackrel{u \in V(G)}{u \sim_G v}} l(u)$. Then
\begin{align*}
w([v]) &= \displaystyle \sum_{[u]\in N([v])} l'([u])\\
 &= \displaystyle \sum \big( \sum_{\stackrel{u'\in V(G)}{u' \sim_G u}}l(u')\big)\\
 &= \sum_{u \in N(v)}l(u)\\
 &= w(v)
\end{align*}
Since $G$ is $A$-vertex magic, we have $H$ is $A'$ magic.\\
Conversely, assume that the reduced graph $H$ of $G$ is $A'$ magic and the corresponding labeling be $f$. Define $f' : V(G) \to A\setminus \{0\}$ by if $|[v]| =1$, then $f'(v)=f([v])$ and if $|[v]| \geq 2$, then we label the vertices of the class using Lemma $\ref{BP2011}$ in such a way that sum of the labels of all vertices in this equivalence class is equal to $f([v])$ in the reduced graph $H$ of G. 
\begin{align*}
  w(v) & = \sum_{u \in N(v)}f'(u)\\
 &= \displaystyle \sum_{[u]\in N([v])} f([u])\\
 &=  w([v]).
\end{align*}
\end{proof}

The following corollaries are immediate consequence of Theorem $\ref{BP2009}$. 

\begin{corollary}\label{BP2014}
Let $A$ be an Abelian group with at least three elements. If $G$ is a graph in which every non-pendant vertex is a strong support vertex, then $G$ is $A$-vertex magic.  
\end{corollary}
\begin{proof}
Let $v_1,v_2,\ldots,v_k$ be non-pendant vertices of $G$ and $v_{i_j}$ be the pendant vertices of $G$, which are adjacent to $v_i$, $j = 1,2,\ldots,m_i$, where $m_i \geq 2$ for all $i=1,2,\ldots,k$. Let $H$ be the reduced graph of $G$. Define $l : V(H) \to A\setminus \{0\}$ by \[l([v_m])= \begin{cases} \text{$a$} &\quad\text{if $d_G(v_m) > 2$}\\ \text{$-(d_H([v_i])-1)a$} &\quad\text{if $v_m = v_{i_j}$ for some $i,j$} \end{cases}\]where $a \in A \setminus \{0\}$. Thus $w([v])=a$, for all $[v]\in V(H)$. By Theorem $\ref{BP2009}$, we get the required result.
\end{proof}

The following result proved in \cite{Sabeel} is an immediate consequence of corollary \ref{BP2014}.

\begin{corollary}(Theorem $2.2$ in \cite{Sabeel}) Let $A$ be a finite Abelian group with $|A| \geq 3$. If $G$ is a graph in which every non-pendant vertex is a strong support vertex, then $G$ is $A$-vertex magic.
\end{corollary}

\begin{corollary}
A caterpillar $T$ is $A$-vertex magic, where $|A|>2$ if and only if $T$ has no vertex of  degree two.
\end{corollary}
\begin{proof}
Let $T$ be a caterpillar. Assume that $T$ has a  non-corner vertex with degree $2$,  let $v_2$ be such a vertex. Also, in the reduced graph $H$, $deg([v_2])=2$(see Figure \ref{Fig1}).

\begin{figure}[htbp]
\begin{center}
\unitlength 1mm 
\linethickness{0.4pt}
\ifx\plotpoint\undefined\newsavebox{\plotpoint}\fi 
\begin{picture}(130.823,35.525)(0,0)
\put(90.964,32.112){\circle*{2.623}}
\multiput(4.601,32.148)(.0322143,.0481429){14}{\line(0,1){.0481429}}
\put(5.052,32.822){\line(-1,0){.225}}
\multiput(5.052,32.148)(-.0338,.0337){20}{\line(-1,0){.0338}}
\multiput(4.376,32.822)(.0965,-.0320714){14}{\line(1,0){.0965}}
\put(4.562,32.108){\circle*{2.623}}
\multiput(90.775,32.148)(.0322857,.0481429){14}{\line(0,1){.0481429}}
\put(91.227,32.822){\line(-1,0){.225}}
\multiput(91.227,32.148)(-.03375,.0337){20}{\line(-1,0){.03375}}
\multiput(90.552,32.822)(.0962857,-.0320714){14}{\line(1,0){.0962857}}
\put(90.739,32.112){\circle*{2.623}}
\multiput(129.028,31.923)(.0321429,.0482143){14}{\line(0,1){.0482143}}
\put(129.478,32.598){\line(-1,0){.225}}
\multiput(129.478,31.923)(-.0321905,.0321429){21}{\line(-1,0){.0321905}}
\multiput(128.802,32.598)(.0964286,-.0321429){14}{\line(1,0){.0964286}}
\put(128.988,31.886){\circle*{2.623}}
\thicklines
\put(91.605,32.078){\line(1,0){.9787}}
\put(93.562,32.078){\line(1,0){.9787}}
\put(95.52,32.078){\line(1,0){.9787}}
\put(97.477,32.078){\line(1,0){.9787}}
\put(99.434,32.078){\line(1,0){.9787}}
\put(101.392,32.078){\line(1,0){.9787}}
\put(103.349,32.078){\line(1,0){.9787}}
\put(105.307,32.078){\line(1,0){.9787}}
\put(107.264,32.078){\line(1,0){.9787}}
\put(109.221,32.078){\line(1,0){.9787}}
\put(111.179,32.078){\line(1,0){.9787}}
\put(113.136,32.078){\line(1,0){.9787}}
\put(115.094,32.078){\line(1,0){.9787}}
\put(117.051,32.078){\line(1,0){.9787}}
\put(119.008,32.078){\line(1,0){.9787}}
\put(120.966,32.078){\line(1,0){.9787}}
\put(122.923,32.078){\line(1,0){.9787}}
\put(124.881,32.078){\line(1,0){.9787}}
\put(126.838,32.078){\line(1,0){.9787}}
\put(128.795,32.078){\line(1,0){.9787}}
\thinlines
\multiput(128.575,32.598)(.032143,-.032143){7}{\line(1,0){.032143}}
\put(90.325,31.925){\line(-1,0){85.95}}
\put(4.827,32.15){\line(0,-1){24.75}}
\put(5.012,7.812){\circle*{2.623}}
\multiput(4.827,8.076)(.032143,-.032286){7}{\line(0,-1){.032286}}
\put(44.389,31.886){\circle*{2.623}}
\multiput(42.627,11.451)(.032286,-.032429){7}{\line(0,-1){.032429}}
\multiput(44.653,8.076)(.032143,-.032286){7}{\line(0,-1){.032286}}
\multiput(91.902,8.524)(.032286,-.031857){7}{\line(1,0){.032286}}
\put(90.552,32.15){\line(0,-1){24.75}}
\put(90.74,7.812){\circle*{2.623}}
\multiput(90.552,8.076)(.032286,-.032286){7}{\line(0,-1){.032286}}
\put(129.252,32.15){\line(0,-1){24.75}}
\put(129.435,7.812){\circle*{2.623}}
\multiput(129.252,8.076)(.032286,-.032286){7}{\line(0,-1){.032286}}
\put(4.375,35.525){\makebox(0,0)[cc]{$[v_1]$}}
\put(44.2,35.525){\makebox(0,0)[cc]{$[v_2]$}}
\put(90.55,35.075){\makebox(0,0)[cc]{$[v_3]$}}
\put(128.8,35.3){\makebox(0,0)[cc]{$[v_k]$}}
\put(4.5,2.5){\makebox(0,0)[cc]{$[v_1']$}}
\put(90.25,2.75){\makebox(0,0)[cc]{$[v_3']$}}
\put(129.999,3.25){\makebox(0,0)[cc]{$[v_k']$}}
\end{picture}
\end{center}
\caption{}\label{Fig1}
\end{figure}

Suppose $H$ is $A'$-vertex magic, then $w([v_2])= l([v_1])+l([v_3]).$ Let $[v_1']$ be a equivalence class which contained all vertices only adjacent to $v_1$ in $T$. Since $w([v_1'])=w([v_2])$ which implies $l([v_3]) = 0$. But we know $|[v_3]|=1$, by Theorem \ref{BP2009} which is a contradiction to $l([v_3]) = 0$ in $H$. The same argument works for corner vertex also.

Conversely,
let $T$ be a caterpillar such that it has no vertex of degree $2$. The reduced graph $H$ of $T$ is $A'$ magic, for all Abelian groups $A$ with $|A| > 2$ (see Figure \ref{Fig2}).  By Theorem \ref{BP2009}, $T$ is $A$-vertex magic.
\end{proof}

\begin{figure}[htbp]
\begin{center}
\unitlength 1mm 
\linethickness{0.4pt}
\ifx\plotpoint\undefined\newsavebox{\plotpoint}\fi 
\begin{picture}(141.75,36.5)(0,0)
\put(97.459,32.706){\circle*{2.915}}
\multiput(1.502,32.747)(.0334,.05){15}{\line(0,1){.05}}
\put(2.003,33.497){\line(-1,0){.25}}
\multiput(2.003,32.747)(-.0326522,.0326087){23}{\line(-1,0){.0326522}}
\multiput(1.252,33.497)(.1000667,-.0333333){15}{\line(1,0){.1000667}}
\put(1.458,32.704){\circle*{2.915}}
\multiput(97.253,32.747)(.0333333,.05){15}{\line(0,1){.05}}
\put(97.753,33.497){\line(-1,0){.25}}
\multiput(97.753,32.747)(-.0326087,.0326087){23}{\line(0,1){.0326087}}
\multiput(97.003,33.497)(.1,-.0333333){15}{\line(1,0){.1}}
\put(97.209,32.706){\circle*{2.915}}
\multiput(139.752,32.497)(.0333333,.05){15}{\line(0,1){.05}}
\put(140.252,33.247){\line(-1,0){.25}}
\multiput(140.252,32.497)(-.0326087,.0326087){23}{\line(0,1){.0326087}}
\multiput(139.502,33.247)(.1,-.0333333){15}{\line(1,0){.1}}
\put(139.709,32.456){\circle*{2.915}}
\thicklines
\put(98.181,32.677){\line(1,0){.9886}}
\put(100.158,32.677){\line(1,0){.9886}}
\put(102.136,32.677){\line(1,0){.9886}}
\put(104.113,32.677){\line(1,0){.9886}}
\put(106.09,32.677){\line(1,0){.9886}}
\put(108.067,32.677){\line(1,0){.9886}}
\put(110.045,32.677){\line(1,0){.9886}}
\put(112.022,32.677){\line(1,0){.9886}}
\put(113.999,32.677){\line(1,0){.9886}}
\put(115.976,32.677){\line(1,0){.9886}}
\put(117.953,32.677){\line(1,0){.9886}}
\put(119.931,32.677){\line(1,0){.9886}}
\put(121.908,32.677){\line(1,0){.9886}}
\put(123.885,32.677){\line(1,0){.9886}}
\put(125.862,32.677){\line(1,0){.9886}}
\put(127.84,32.677){\line(1,0){.9886}}
\put(129.817,32.677){\line(1,0){.9886}}
\put(131.794,32.677){\line(1,0){.9886}}
\put(133.771,32.677){\line(1,0){.9886}}
\put(135.749,32.677){\line(1,0){.9886}}
\put(137.726,32.677){\line(1,0){.9886}}
\put(139.703,32.677){\line(1,0){.9886}}
\thinlines
\multiput(139.25,33.247)(.03125,-.03125){8}{\line(0,-1){.03125}}
\put(96.753,32.5){\line(-1,0){95.5}}
\put(1.753,32.75){\line(0,-1){27.5}}
\put(1.958,5.707){\circle*{2.915}}
\multiput(1.753,5.999)(.03125,-.03125){8}{\line(0,-1){.03125}}
\put(45.711,32.456){\circle*{2.915}}
\multiput(43.754,9.749)(.03125,-.03125){8}{\line(0,-1){.03125}}
\put(46.004,32.75){\line(0,-1){27.5}}
\put(46.209,5.707){\circle*{2.915}}
\multiput(46.004,5.999)(.03125,-.03125){8}{\line(0,-1){.03125}}
\multiput(98.505,6.499)(.03125,-.03125){8}{\line(0,-1){.03125}}
\put(97.005,32.75){\line(0,-1){27.5}}
\put(97.211,5.707){\circle*{2.915}}
\multiput(97.005,5.999)(.03125,-.03125){8}{\line(0,-1){.03125}}
\put(140.002,32.75){\line(0,-1){27.5}}
\put(140.209,5.707){\circle*{2.915}}
\multiput(140.002,5.999)(.03125,-.03125){8}{\line(0,-1){.03125}}
\put(1.251,36.5){\makebox(0,0)[cc]{$a$}}
\put(45.501,36.5){\makebox(0,0)[cc]{$a$}}
\put(97.001,36){\makebox(0,0)[cc]{$a$}}
\put(139.5,36.25){\makebox(0,0)[cc]{$a$}}
\put(1.501,1){\makebox(0,0)[cc]{$0$}}
\put(46.501,.75){\makebox(0,0)[cc]{$-a$}}
\put(97.251,1){\makebox(0,0)[cc]{$-a$}}
\put(140.25,0){\makebox(0,0)[cc]{$0$}}
\end{picture}
\end{center}
\caption{$A'$-vertex magic labeling of reduced graph $H$ of $T$}\label{Fig2}
\end{figure}
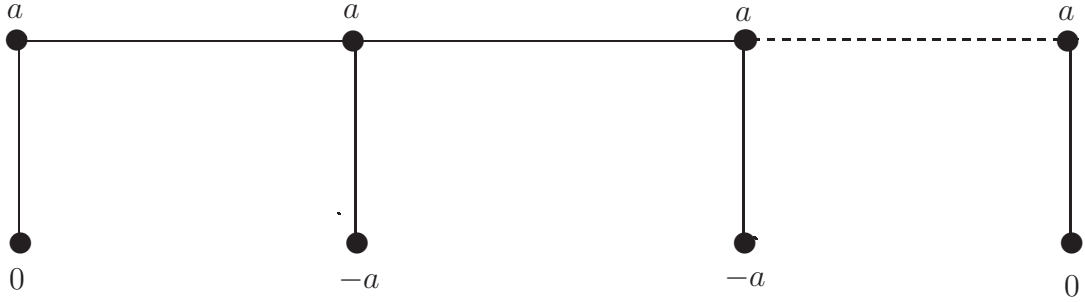

\begin{remark}
Given a graph $G$, under the relation $\sim_G$ the vertex set $V(G)$ is partitioned into equivalence classes. For every equivalence class $[v_i]$, with $|[v_i]| = 1$, introduce a new vertex $x_i$ in  $[v_i]$. Now, join the new vertex $x_i$ to all vertices in equivalence classes belongs to $N([v_i])$. Thus each equivalence class will have at least two vertices. Hence by Theorem $\ref{BP2008}$ the resultant graph is $A$-vertex magic, for all $|A|>2$.
\end{remark}

In the above construction we newly add at most $|V(G)|$ vertices, the graph $G$ is convert into $A$-vertex magic, for all $|A|>2$. Therefore, we get the following corollary as an  immediate consequence of the above construction.
\begin{corollary}\label{BP2015}
Any graph $H'$ is an induced subgraph of $A$-vertex magic graph $G$, where $|A| > 2$.
\end{corollary}
The above corollary \ref{BP2015} is a generalisation of the following result proved in \cite{Sabeel} which has been proved only for finite Abelian groups and it uses more number of vertices. 

\begin{corollary}(Corollary 2.3, \cite{Sabeel})
Any graph $G$ is an induced subgraph of an $A$-vertex magic graph $H$.
\end{corollary}

We now state a procedure to
construct infinite classes of $A$-vertex magic graphs, when $|A|>2$. Let $G$ be a $A$-vertex magic graph  with magic labeling $l$. As $\sim_G$ is an equivalence relation, join a new vertex $x$ in any one of the equivalence class $[v_i]$ of $G$ and join by an edge $x$ into each vertex in $N(v_i)$, the resulting graph is $G'$. Then by Lemma \ref{BP2011}, we can relabel the vertices in $[v_i]$ and $x$ whose label sum is equal to $\sum_{v \in [v_i]} l({v})$, for the other vertices in $G'$ we retain the labeling given in $G$ by $l$. 

The following result proved in \cite{Sabeel} 
is an immediate consequence of the above construction.

\begin{corollary}\label{BP2016}(Theorem 3.1, \cite{Sabeel})
Let $t \geq 2$. Suppose $G$ is an $A$-vertex magic graph of order $n$ with a labeling $l$ and magic constant $g$. If there exists an edge $uv$ in $G$ with $l(u)=l(v)=g$, then the graph $G'$, obtained from $G$ by subdividing the edge $uv$ and by attaching $t$ pendant vertices at the new vertex $x$, is an $A$-vertex magic graph of order $n+t+1$ with the same magic constant $g$.
\end{corollary}

The main difference between our construction  and corollary \ref{BP2016} is, it deals with more number of vertices and it is proved under an additional assumption on the magic labeling. We just add only one new vertex and our construction is true for infinite Abelian groups also.

\section{Group Vertex Magic labeling of Generalised Friendship Graph} 

In \cite{Fernaua}, Henning Fernau et al. defined the Generalised Friendship Graph as follows
\begin{definition} 
The friendship graph $f_m$ is a collection of $m$ triangles with a common vertex. The generalised friendship graph $f_{n,m}$ is a collection of $m$ cycles (all of order $n$), meeting at a common vertex.
\end{definition}

\begin{proposition}\label{BP2012}
Let $m\in \mathbb{N}$. The group $A$ is isomorphic to the group $\prod \mathbb{Z}_{m_i}$. Now, $m_i \mid (2m-1),$ for all $i$ if and only if $o(a) \mid (2m-1)$, for all $a \in A$. 
\end{proposition}
\begin{proof}
Let $A \cong \prod \mathbb{Z}_{m_i}$. Suppose that  $m_i \mid (2m-1)$, for all $i$. Let $a \in A$. Then $a=(a_1,a_2,\ldots,a_i,\ldots),$ where $a_i \in \mathbb{Z}_{m_i}$ for each $i$, which implies $o(a_i) \mid m_i$ for all $i$.
Then $o(a_i) \mid (2m-1).$ Now, $(2m-1)a_i = 0$ for all $i$. Therefore, $(2m-1)a=(0,0,\ldots,0,\ldots)$.
Hence $o(a) \mid (2m-1)$.\\ Conversely, assume that $o(a) \mid (2m-1),$ for all $a \in A$. Since $(2m-1)a=(0,0,\ldots,0,\ldots)$, which implies $(2m-1)a_i = 0$ for all $i$. Therefore $o(a_i) \mid (2m-1)$. Since this is true for every $a_i \in \mathbb{Z}_{m_i},$ it also hold for generator of $\mathbb{Z}_{m_i}$. Hence $m_i \mid (2m-1)$, for all $i$.

\end{proof}

\begin{theorem}\label{BP1118}
Let $n \equiv 1,3(\textrm{mod}\ 4).$ The generalised friendship graph $G=f_{n,m}$ is $A$-vertex magic if and only if $A \ncong \prod \mathbb{Z}_{m_i},$ where $m_i \mid (2m-1),$ for all $i$.
\end{theorem}
\begin{proof}
Let $V_1,V_2,\ldots,V_{m+1}$ be a partition of $V(G)$ with $|V_i|=n-1$, where $i=1,2,\ldots,m$. Let $V_i=\{v_{i_1},v_{i_2},\ldots,v_{i_{n-1}}\}$ and $V_{m+1}=\{v\}$. If $G$ is $A$-vertex magic, then $l(v_{i_k}) = l(v_{i_{k+4}})$. 
\item[\bf Case 1:]$n=3$.\\
Assume that $f_{3,m}$ is $A$-vertex magic with $l(v_{i_1})=a, l(v_{i_2})=b, l(v)=c$. Since $w(v_{i_1})=w(v_{i_2})$, which implies $a=b$. Let $j=\{1,2,\ldots,m\}$ and $i \neq j$. Assume that $l(v_{j_1})=a_1, l(v_{j_2})=b_1$.  Since $w(v_{i_1})=w(v_{j_1})$, which implies $b=b_1$. Since $w(v_{i_2})=w(v_{j_2})$, which implies $a=a_1$. Since $w(v_{i_1})=w(v_{j_2}),$ which implies $b=a_1$. Also, $w(v_{i_1})=w(v)$, which implies $c=(2m-1)a$. Hence $A \ncong \prod \mathbb{Z}_{m_i},$ where $m_i \mid (2m-1),$ for all $i$.\\
Conversely, assume that $A \ncong \prod \mathbb{Z}_{m_i},$ where $m_i \mid (2m-1),$ for all $i$. By Proposition \ref{BP2012}, $A$ has an element $a$ such that $o(a) \nmid (2m-1)$. Define $l: V(G) \to A \setminus\{0\}$ by $l(v_{i_k})=a$, for all $i,k$ and $l(v)=(2m-1)a$. Thus $w(v)= 2ma$, for all $v \in V(G)$.
\item[\bf Case 2:]$n \equiv 3(\textrm{mod}\ 4)$ and $n>3$.\\
Assume that $f_{n,m}$ is $A$-vertex magic with $l(v_{i_1})=a,l(v_{i_2})=b,l(v_{i_3})=c,l(v_{i_4})=d$ and $l(v)=e$. Since $w(v_{i_1})=w(v_{i_{n-1}})$, which implies $b=a$. Since $w(v_{i_1})=w(v_{i_3})$, which implies $e=d$. Also $w(v_{i_2})=w( v_{i_{n-1}})$, which implies $c=e$. Let $j \in \{1,2,\ldots,m\}$ and $i \neq j$. Assume that $l(v_{j_1})=a_1,l(v_{j_2})=b_1,l(v_{j_3})=c_1,l(v_{j_4})=d_1$. By similar argument, we get $b_1=a_1$ and $c_1=d_1=e$. Since $w(v_{i_1})=w(v_{j_1})$, which implies $b=b_1$. Since $w(v_{i_{n-1}})=w(v_{j_{n-1}}),$ which implies $a=a_1$. Now, $w(v_{i_1})=w(v)$, which implies $e=(2m-1)a$. Hence $A \ncong \prod \mathbb{Z}_{m_i},$ where $m_i \mid (2m-1),$ for all $i$. \\
Conversely, assume that $A \ncong \prod \mathbb{Z}_{m_i},$ where $m_i \mid (2m-1),$ for all $i$. By Proposition \ref{BP2012}, $A$ has an element $a$ such that $o(a) \nmid (2m-1)$. Define $l: V(G) \to A \setminus\{0\}$ by \[l(v_{i_k}) = \begin{cases} \text{$(2m-1)a$} &\quad\text{if $k \equiv 0,3 (\textrm{mod}\ 4)$}\\ \text{$a$} &\quad\text{otherwise,}\\ \end{cases} \] for all $i,k$ and $l(v)=(2m-1)a$. Thus $w(v)= 2ma$, for all $v \in V(G)$.
\item[\bf Case 3:]$n\equiv 1(\textrm{mod}\ 4)$.\\
Assume that $f_{n,m}$ is $A$-vertex magic with $l(v_{i_1})=a,l(v_{i_2})=b,l(v_{i_3})=c,l(v_{i_4})=d$ and $l(v)=e$. Since $w(v_{i_1})=w(v_{i_{n-1}})$, which implies $b=c$. Since $w(v_{i_2})=w(v_{i_3})$, which implies $a=d$. Also $w(v_{i_2})=w(v_{i_{n-1}})$, which implies $a=e$. Let $j \in \{1,2,\ldots,m\}$ and $i \neq j$. Assume that $l(v_{j_1}) = a_1, l(v_{i_2}) =b_1, l(v_{i_3})= c_1, l(v_{i_4}) = d_1$. By similar argument, we get $b_1=c_1, a_1=d_1=e$,which implies $a_1=a=d=d_1$. Since $w(v_{i_1})=w(v_{j_1})$, which implies $b=b_1$. Since $w(v_{i_{n-1}})= w(v_{j_{n-1}})$, which implies $c=c_1$. 
Also, $w(v)=w(v_{i_1})$, which implies $b=(2m-1)a$. Hence $A \ncong \prod \mathbb{Z}_{m_i},$ where $m_i \mid (2m-1),$ for all $i$.\\
Conversely, assume that $A \ncong \prod \mathbb{Z}_{m_i},$ where $m_i \mid (2m-1)$. By Proposition \ref{BP2012}, $A$ has an element $a$ such that $o(a) \nmid (2m-1)$. Define $l: V(G) \to A \setminus\{0\}$ by \[l(v_{i_k}) = \begin{cases} \text{$(2m-1)a$} &\quad\text{if $k \equiv 2,3 (\textrm{mod}\ 4)$}\\ \text{$a$} &\quad\text{otherwise,}\\ \end{cases} \] for all $i,k$ and $l(v)=a$. Thus $w(v)= 2ma,$ for all $v \in V(G)$.
\end{proof}

\begin{theorem}\label{BP1119}
If $n \equiv 0 (\textrm{mod}\ 4)$, then the generalised friendship graph $f_{n,m}$ is group vertex magic.
\end{theorem}
\begin{proof}
Let A be a non-trivial Abelian group. 
Let $V_1,V_2,\ldots,V_{m+1}$ be a partition of $V(G)$ with $|V_i| = n-1$, where $i=1,2,\ldots,m$. Let $V_i=\{v_{i_1},v_{i_2},\ldots,v_{i_{n-1}}\}$ and $V_{m+1}= \{v\}$. Define $l : V(G) \to A \setminus \{0\}$ by \[ l(v_{i_k}) = \begin{cases} \text{$-a$} &\quad\text{if $k \equiv 1,2 (\textrm{mod}\ 4)$}\\ \text{$a$} &\quad\text{otherwise}\\ \end{cases} \] and $l(v)= a$. Thus $w(v) = 0,$ for all $v \in V(G)$ (see Figure \ref{BP2fig4}).

\begin{figure}[htbp]
\begin{center}
\unitlength 1mm 
\linethickness{0.4pt}
\ifx\plotpoint\undefined\newsavebox{\plotpoint}\fi 
\begin{picture}(101.6,88.674)(0,0)
\put(29.121,11.995){\circle*{2.04}}
\put(29.121,75.52){\circle*{2.04}}
\put(63.07,12.169){\circle*{2.04}}
\put(63.07,75.694){\circle*{2.04}}
\put(77.246,12.169){\circle*{2.04}}
\put(77.246,75.694){\circle*{2.04}}
\put(29.15,11.851){\line(1,0){34.3}}
\put(29.15,75.377){\line(1,0){34.3}}
\put(37.696,11.995){\circle*{2.04}}
\put(37.696,75.52){\circle*{2.04}}
\put(45.571,11.995){\circle*{2.04}}
\put(45.571,75.52){\circle*{2.04}}
\put(54.321,12.169){\circle*{2.04}}
\put(54.321,75.694){\circle*{2.04}}
\put(63.38,12.131){\line(1,0){.9334}}
\put(65.247,12.154){\line(1,0){.9334}}
\put(67.113,12.177){\line(1,0){.9334}}
\put(68.98,12.201){\line(1,0){.9334}}
\put(70.847,12.224){\line(1,0){.9334}}
\put(72.714,12.247){\line(1,0){.9334}}
\put(74.581,12.271){\line(1,0){.9334}}
\put(76.447,12.294){\line(1,0){.9334}}
\put(63.38,75.656){\line(1,0){.9334}}
\put(65.247,75.679){\line(1,0){.9334}}
\put(67.113,75.703){\line(1,0){.9334}}
\put(68.98,75.726){\line(1,0){.9334}}
\put(70.847,75.75){\line(1,0){.9334}}
\put(72.714,75.773){\line(1,0){.9334}}
\put(74.581,75.797){\line(1,0){.9334}}
\put(76.447,75.82){\line(1,0){.9334}}
\put(16.346,69.632){\circle*{2.04}}
\put(90.546,70.506){\circle*{2.04}}
\put(16.521,35.682){\circle*{2.04}}
\put(90.721,36.557){\circle*{2.04}}
\put(16.521,21.505){\circle*{2.04}}
\put(90.721,22.38){\circle*{2.04}}
\put(16.202,69.601){\line(0,-1){34.3}}
\put(90.403,70.476){\line(0,-1){34.3}}
\put(16.346,61.056){\circle*{2.04}}
\put(90.546,61.931){\circle*{2.04}}
\put(16.346,53.181){\circle*{2.04}}
\put(90.546,54.056){\circle*{2.04}}
\put(16.521,44.431){\circle*{2.04}}
\put(90.721,45.305){\circle*{2.04}}
\put(16.482,35.231){\line(0,-1){1}}
\put(16.507,33.231){\line(0,-1){1}}
\put(16.532,31.231){\line(0,-1){1}}
\put(16.557,29.231){\line(0,-1){1}}
\put(16.582,27.231){\line(0,-1){1}}
\put(16.607,25.231){\line(0,-1){1}}
\put(16.632,23.231){\line(0,-1){1}}
\put(90.683,36.107){\line(0,-1){.9334}}
\put(90.706,34.24){\line(0,-1){.9334}}
\put(90.729,32.373){\line(0,-1){.9334}}
\put(90.753,30.506){\line(0,-1){.9334}}
\put(90.776,28.64){\line(0,-1){.9334}}
\put(90.799,26.773){\line(0,-1){.9334}}
\put(90.823,24.906){\line(0,-1){.9334}}
\put(90.846,23.039){\line(0,-1){.9334}}
\put(51.521,45.422){\circle*{2.04}}
\multiput(16.551,69.6)(.04814305365,-.03369876204){727}{\line(1,0){.04814305365}}
\put(51.551,45.101){\line(0,-1){.175}}
\multiput(51.551,44.926)(-.04992867332,-.03370185449){701}{\line(-1,0){.04992867332}}
\multiput(29.151,75.55)(.03369402985,-.04570746269){670}{\line(0,-1){.04570746269}}
\multiput(51.726,44.926)(.03371559633,.04013630406){763}{\line(0,1){.04013630406}}
\multiput(51.376,44.751)(.05038560411,.03373907455){778}{\line(1,0){.05038560411}}
\multiput(51.201,45.801)(.031818,-.047818){11}{\line(0,-1){.047818}}
\multiput(28.801,11.676)(.03371323529,.04889705882){680}{\line(0,1){.04889705882}}
\multiput(51.726,44.926)(.03370712401,-.04317282322){758}{\line(0,-1){.04317282322}}
\multiput(90.576,22.7)(-.0585090361,.0337364458){664}{\line(-1,0){.0585090361}}
\put(92.85,44.837){\oval(17.5,58.275)[]}
\put(42.975,36.874){\framebox(17.325,16.8)[cc]{}}
\put(51.55,38.974){\makebox(0,0)[cc]{$V_{m+1}$}}
\put(93.9,69.949){\makebox(0,0)[cc]{$-a$}}
\put(94.25,61.025){\makebox(0,0)[cc]{$-a$}}
\put(94.25,53.85){\makebox(0,0)[cc]{$a$}}
\put(94.075,44.4){\makebox(0,0)[cc]{$a$}}
\put(94.425,36.525){\makebox(0,0)[cc]{$-a$}}
\put(94.6,22.175){\makebox(0,0)[cc]{$a$}}
\put(28.45,79.049){\makebox(0,0)[cc]{$-a$}}
\put(45.25,79.224){\makebox(0,0)[cc]{$a$}}
\put(63.1,78.875){\makebox(0,0)[cc]{$-a$}}
\put(54.525,79.224){\makebox(0,0)[cc]{$a$}}
\put(37.375,79.224){\makebox(0,0)[cc]{$-a$}}
\put(77.275,79.049){\makebox(0,0)[cc]{$a$}}
\put(28.625,8.7){\makebox(0,0)[cc]{$-a$}}
\put(45.425,8.875){\makebox(0,0)[cc]{$a$}}
\put(63.275,8.525){\makebox(0,0)[cc]{$-a$}}
\put(54.7,8.875){\makebox(0,0)[cc]{$a$}}
\put(37.55,8.875){\makebox(0,0)[cc]{$-a$}}
\put(77.45,8.7){\makebox(0,0)[cc]{$a$}}
\put(49.1,4.5){\makebox(0,0)[cc]{$V_m$}}
\put(6.4,46.325){\makebox(0,0)[cc]{$V_1$}}
\put(97.575,47.9){\makebox(0,0)[cc]{$V_3$}}
\put(50.325,84.999){\makebox(0,0)[cc]{$V_2$}}
\put(11.5,68.698){\makebox(0,0)[cc]{$-a$}}
\put(11.85,59.774){\makebox(0,0)[cc]{$-a$}}
\put(11.85,52.599){\makebox(0,0)[cc]{$a$}}
\put(11.675,43.149){\makebox(0,0)[cc]{$a$}}
\put(12.025,35.274){\makebox(0,0)[cc]{$-a$}}
\put(12.2,20.924){\makebox(0,0)[cc]{$a$}}
\put(11.125,44.299){\oval(17.75,59.75)[]}
\put(53.625,79.799){\oval(57.25,17.75)[]}
\put(53.875,9.5){\oval(57.75,15)[]}
\put(51.5,50.25){\makebox(0,0)[cc]{$a$}}
\multiput(67.18,27.68)(.0409664,.032563){17}{\line(1,0){.0409664}}
\multiput(68.573,28.787)(.0409664,.032563){17}{\line(1,0){.0409664}}
\multiput(69.965,29.894)(.0409664,.032563){17}{\line(1,0){.0409664}}
\multiput(71.358,31.001)(.0409664,.032563){17}{\line(1,0){.0409664}}
\end{picture}
\end{center}
\caption{Group vertex magic labeling of $f_{n,m}$, where $n \equiv 0 (\textrm{mod}\ 4)$ } \label{BP2fig4}
\end{figure}
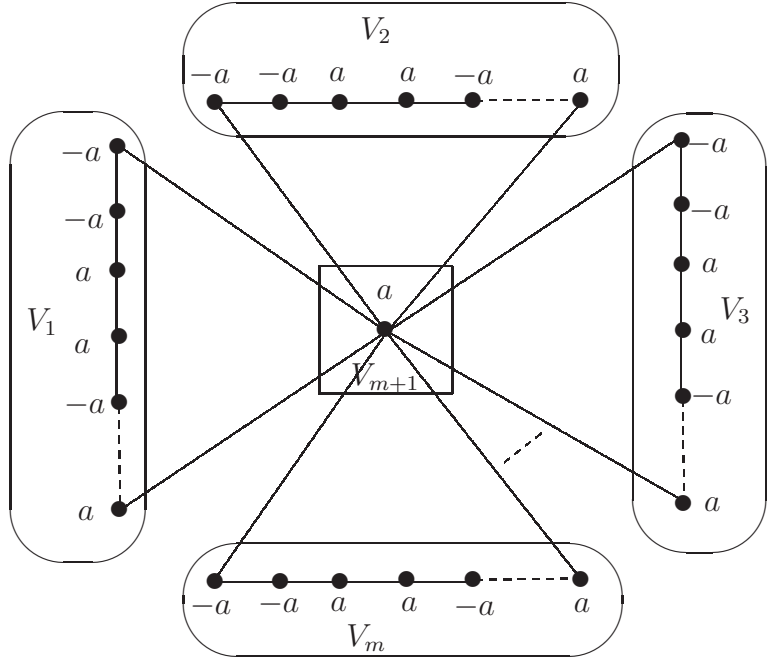
\end{proof}

\begin{theorem}\label{BP2010}
Let $n \equiv 2(\textrm{mod}\ 4)$. The generalised friendship graph $G=f_{n,m}$ is group vertex magic if and only if $m \equiv 1(\textrm{mod}\ 3)$. 
\end{theorem}
\begin{proof}
Let $V_1,V_2,\ldots,V_{m+1}$ be a partition of $V(G)$ with $|V_i|=n-1$, where $i=1,2,\ldots,m$. Let $V_i=\{v_{i_1},v_{i_2},\ldots,v_{i_{n-1}}\}$ and $V_{m+1}=\{v\}$. Assume that $f_{n,m}$ is group vertex magic. By observation, we have $l(v_{i_k}) = l(v_{i_{k+4}})$. Assume that $l(v_{i_1})=a,l(v_{i_2})=b,l(v_{i_3})=c,l(v_{i_4})=d$ and $l(v)=e$. Since $w(v_{i_1})= w(v_{i_{n-1}})$, which implies $b=d$. Since $w(v_{i_1})=w(v_{i_3}),$ which implies $e=d$. Also, $w(v_{i_2})=w(v_{i_3}),$ which implies $a+c=b+d$. Let $j \in \{1,2,\ldots,m\}$ and $i \neq j$. Assume that $l(v_{j_1})=a_1,l(v_{j_2})=b_1,l(v_{j_3})=c_1,l(v_{j_4})=d_1$. By similar argument, we get $b_1=d_1=e$ and $a_1+c_1=b_1+d_1$. Since $w(v_{i_1})=w(v_{j_1})$, which implies $b=b_1$. Since $w(v_{i_{n-1}})=w(v_{j_{n-1}})$, which implies $d=d_1$. Let $a'\in \mathbb{Z}_3 \setminus \{0\}$. Assume that $l(v_{i_2})=l(v_{i_4})=l(v)=a'$. Since $w(v_{i_2})=w(v_{i_3}),$  which implies $l(v_{i_1})=l(v_{i_3})=a'$, therefore $l(v)=a',$ for all $v$. Since $w(v_{i_1})=w(v)$, which implies $2a'=2ma'$, we get $m \equiv 1 (\textrm{mod}\ 3)$.\\
Conversely, assume that $m\equiv 1(\textrm{mod}\ 3)$ and $p$ is a prime number. 
\item[\bf Case 1.] $p=2$.\\If $2$ divides $o(A)$, then define $l: V(G) \to A\setminus\{0\}$ by $l(v)=a,$ for all $v \in V(G)$, where $o(a)=2$. Thus $w(v)= 0,$ for all $v \in V(G)$. \item[\bf Case 2.] $p=3$.\\ If $3$ divides $o(A)$, then define $l: V(G) \to A\setminus\{0\}$ by $l(v)=a,$ for all $v \in V(G)$, where $o(a)=3$. Thus $w(v)=2a,$ for all $v$.\item[\bf Case 3. ]  $p>3$. \\ If $p$ divides $o(A)$, then\item[\bf Subcase 1.]Assume that $m$ is odd.\\ Define $l: V(G) \to A\setminus\{0\}$ by \[l(v_{i_k})=\begin{cases} \text{$-a$} &\quad\text{if $i$ is even and $k \equiv 1 (\textrm{mod}\ 4)$}\\ \text{3a} &\quad\text{if $i$ is even and $k \equiv 3 (\textrm{mod}\ 4)$}\\ \text{$a$} &\quad\text{otherwise}\\ \end{cases} \]and $l(v)=a$, where $o(a)=p$. Thus $w(v)=2a,$ for all $v$. \item[\bf Subcase 2.]Assume that $m$ is even.\\ Define $l: V(G) \to A\setminus\{0\}$ by \[l(v_{i_k})=\begin{cases} \text{3a} &\quad\text{if $i=1$ and $k \equiv 1(\textrm{mod}\ 4)$}\\ \text{$-a$} &\quad\text{if $i=1$ and $k \equiv 3(\textrm{mod}\ 4)$}\\ \text{$-2a$} &\quad\text{if $i=2$ and $k \equiv 1(\textrm{mod}\ 4)$}\\ \text{$4a$} &\quad\text{if $i=2$ and $k \equiv 3(\textrm{mod}\ 4)$}\\ \text{$-a$} &\quad\text{if $i$ is even, $i>2$ and $k \equiv 1(\textrm{mod}\ 4)$}\\ \text{$3a$} &\quad\text{if $i$ is even, $i>2$ and $k \equiv 3 (\textrm{mod}\ 4)$}\\ \text{$a$} &\quad\text{otherwise}\\ \end{cases} \] and $l(v)=a$, where $o(a)=p$. Thus $w(v)=2a, $ for all $v \in V(G)$ (see Figure \ref{BP2fig5}). 
\\ By Theorem $\ref{BP2002}$ and Case 3, we get $G$ is $\mathbb{Z}$-vertex magic.
\end{proof}

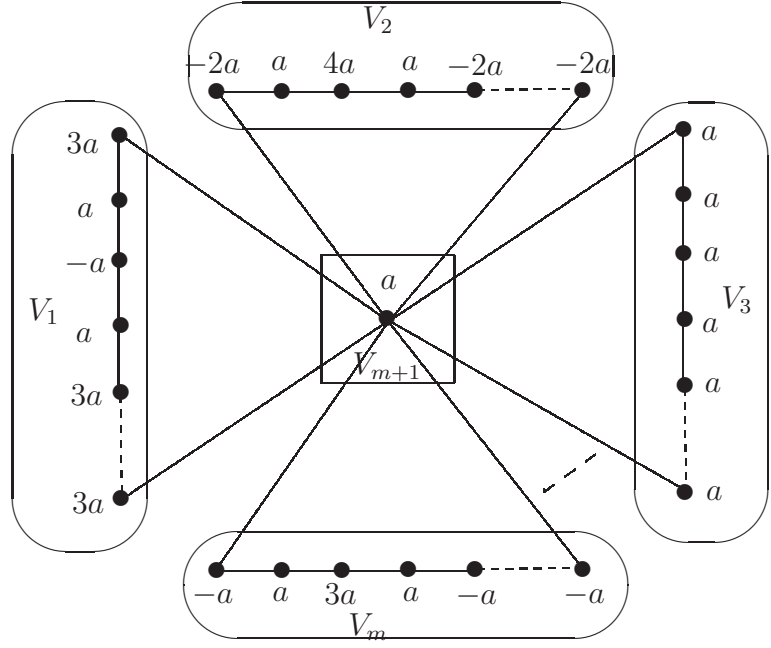
\begin{figure}[htbp]
\begin{center}
\unitlength 1mm 
\linethickness{0.4pt}
\ifx\plotpoint\undefined\newsavebox{\plotpoint}\fi 
\begin{picture}(107.1,85.851)(0,0)
\put(34.621,10.57){\circle*{2.04}}
\put(34.621,74.096){\circle*{2.04}}
\put(68.57,10.744){\circle*{2.04}}
\put(68.57,74.269){\circle*{2.04}}
\put(82.746,10.744){\circle*{2.04}}
\put(82.746,74.269){\circle*{2.04}}
\put(34.65,10.427){\line(1,0){34.3}}
\put(34.65,73.953){\line(1,0){34.3}}
\put(43.196,10.57){\circle*{2.04}}
\put(43.196,74.096){\circle*{2.04}}
\put(51.071,10.57){\circle*{2.04}}
\put(51.071,74.096){\circle*{2.04}}
\put(59.821,10.744){\circle*{2.04}}
\put(59.821,74.269){\circle*{2.04}}
\put(68.88,10.706){\line(1,0){.9334}}
\put(70.747,10.729){\line(1,0){.9334}}
\put(72.613,10.753){\line(1,0){.9334}}
\put(74.48,10.776){\line(1,0){.9334}}
\put(76.347,10.8){\line(1,0){.9334}}
\put(78.214,10.823){\line(1,0){.9334}}
\put(80.081,10.847){\line(1,0){.9334}}
\put(81.947,10.87){\line(1,0){.9334}}
\put(68.88,74.231){\line(1,0){.9334}}
\put(70.747,74.254){\line(1,0){.9334}}
\put(72.613,74.278){\line(1,0){.9334}}
\put(74.48,74.301){\line(1,0){.9334}}
\put(76.347,74.325){\line(1,0){.9334}}
\put(78.214,74.348){\line(1,0){.9334}}
\put(80.081,74.372){\line(1,0){.9334}}
\put(81.947,74.395){\line(1,0){.9334}}
\put(21.846,68.207){\circle*{2.04}}
\put(96.046,69.081){\circle*{2.04}}
\put(22.021,34.258){\circle*{2.04}}
\put(96.221,35.133){\circle*{2.04}}
\put(22.021,20.081){\circle*{2.04}}
\put(96.221,20.956){\circle*{2.04}}
\put(21.702,68.176){\line(0,-1){34.3}}
\put(95.903,69.051){\line(0,-1){34.3}}
\put(21.846,59.632){\circle*{2.04}}
\put(96.046,60.506){\circle*{2.04}}
\put(21.846,51.756){\circle*{2.04}}
\put(96.046,52.632){\circle*{2.04}}
\put(22.021,43.006){\circle*{2.04}}
\put(96.221,43.88){\circle*{2.04}}
\put(21.982,33.806){\line(0,-1){.9999}}
\put(22.007,31.806){\line(0,-1){.9999}}
\put(22.032,29.806){\line(0,-1){.9999}}
\put(22.057,27.806){\line(0,-1){.9999}}
\put(22.082,25.806){\line(0,-1){.9999}}
\put(22.107,23.806){\line(0,-1){.9999}}
\put(22.132,21.807){\line(0,-1){.9999}}
\put(96.183,34.682){\line(0,-1){.9333}}
\put(96.206,32.815){\line(0,-1){.9333}}
\put(96.229,30.948){\line(0,-1){.9333}}
\put(96.253,29.082){\line(0,-1){.9333}}
\put(96.276,27.215){\line(0,-1){.9333}}
\put(96.299,25.348){\line(0,-1){.9333}}
\put(96.323,23.482){\line(0,-1){.9333}}
\put(96.346,21.615){\line(0,-1){.9333}}
\put(57.021,43.998){\circle*{2.04}}
\multiput(22.051,68.175)(.04814305365,-.03369876204){727}{\line(1,0){.04814305365}}
\put(57.051,43.676){\line(0,-1){.175}}
\multiput(57.051,43.502)(-.04992867332,-.03370185449){701}{\line(-1,0){.04992867332}}
\multiput(34.651,74.125)(.03369402985,-.04570597015){670}{\line(0,-1){.04570597015}}
\multiput(57.226,43.502)(.03371559633,.04013499345){763}{\line(0,1){.04013499345}}
\multiput(56.876,43.326)(.05038560411,.03373907455){778}{\line(1,0){.05038560411}}
\multiput(56.701,44.376)(.031818,-.047818){11}{\line(0,-1){.047818}}
\multiput(34.301,10.252)(.03371323529,.04889705882){680}{\line(0,1){.04889705882}}
\multiput(57.226,43.502)(.03370712401,-.04317414248){758}{\line(0,-1){.04317414248}}
\multiput(96.076,21.276)(-.0585090361,.0337349398){664}{\line(-1,0){.0585090361}}
\put(59.5,8.588){\oval(58.45,14.175)[]}
\put(98.35,43.412){\oval(17.5,58.275)[]}
\put(48.475,35.449){\framebox(17.325,16.8)[cc]{}}
\put(57.05,37.549){\makebox(0,0)[cc]{$V_{m+1}$}}
\put(99.4,68.524){\makebox(0,0)[cc]{$a$}}
\put(99.75,59.6){\makebox(0,0)[cc]{$a$}}
\put(99.75,52.425){\makebox(0,0)[cc]{$a$}}
\put(99.575,42.975){\makebox(0,0)[cc]{$a$}}
\put(99.925,35.1){\makebox(0,0)[cc]{$a$}}
\put(100.1,20.751){\makebox(0,0)[cc]{$a$}}
\put(33.95,77.624){\makebox(0,0)[cc]{$-2a$}}
\put(50.75,77.799){\makebox(0,0)[cc]{$4a$}}
\put(68.6,77.451){\makebox(0,0)[cc]{$-2a$}}
\put(60.025,77.799){\makebox(0,0)[cc]{$a$}}
\put(42.875,77.799){\makebox(0,0)[cc]{$a$}}
\put(82.775,77.624){\makebox(0,0)[cc]{$-2a$}}
\put(34.125,7.276){\makebox(0,0)[cc]{$-a$}}
\put(50.925,7.451){\makebox(0,0)[cc]{$3a$}}
\put(68.775,7.101){\makebox(0,0)[cc]{$-a$}}
\put(60.2,7.451){\makebox(0,0)[cc]{$a$}}
\put(43.05,7.451){\makebox(0,0)[cc]{$a$}}
\put(82.95,7.276){\makebox(0,0)[cc]{$-a$}}
\put(54.6,3.075){\makebox(0,0)[cc]{$V_m$}}
\put(11.9,44.9){\makebox(0,0)[cc]{$V_1$}}
\put(103.075,46.475){\makebox(0,0)[cc]{$V_3$}}
\put(58.887,77.451){\oval(55.825,16.8)[]}
\put(55.825,83.574){\makebox(0,0)[cc]{$V_2$}}
\put(17,67.273){\makebox(0,0)[cc]{$3a$}}
\put(17.35,58.349){\makebox(0,0)[cc]{$a$}}
\put(17.35,51.175){\makebox(0,0)[cc]{$-a$}}
\put(17.175,41.725){\makebox(0,0)[cc]{$a$}}
\put(17.525,33.849){\makebox(0,0)[cc]{$3a$}}
\put(17.7,19.5){\makebox(0,0)[cc]{$3a$}}
\put(16.625,42.874){\oval(17.75,59.75)[]}
\put(77.68,20.93){\line(0,1){.5}}
\multiput(77.68,20.93)(.045,.0333333){30}{\line(1,0){.045}}
\multiput(80.38,22.93)(.045,.0333333){30}{\line(1,0){.045}}
\multiput(83.08,24.93)(.045,.0333333){30}{\line(1,0){.045}}
\put(57.25,48.749){\makebox(0,0)[cc]{$a$}}
\end{picture}
\end{center} 
\caption{$\mathbb{Z}_p$-vertex magic of $f_{n,m}$, $p > 3$, $n \equiv 2(\textrm{mod}\ 4)$ and $m$ is even} \label{BP2fig5}
\end{figure}

\begin{theorem}
Let $n \equiv 2(\textrm{mod}\ 4)$ and $m \not \equiv 1(\textrm{mod}\ 3)$. The generalised friendship graph $f_{n,m}$ is $A$-vertex magic if and only if $A \neq \mathbb{Z}_3$.
\end{theorem}
\begin{proof}
Let $V_1,V_2,\ldots,V_{m+1}$ be a partition of $V(G)$ with $|V_i|=n-1$, where $i=1,2,\ldots,m$. Let $V_i=\{v_{i_1},v_{i_2},\ldots,v_{i_{n-1}}\}$ and $V_{m+1}=\{v\}$.
The necessary condition follows from Theorem \ref{BP2010}. 
Conversely, assume that $A \neq \mathbb{Z}_3$.
\item[\bf Case 1.] $2$ divides $o(A)$.\\
The  proof follows from
the labeling in case $1$ of above theorem.
\item[\bf Case 2.] $9$ divides $o(A)$.\\
By Sylow's first theorem and fundamental theorem of finite Abelian groups, $A$ has a subgroup isomorphic to either $\mathbb{Z}_9$ or $\mathbb{Z}_3 \times \mathbb{Z}_3$. Suppose $A$ has a subgroup isomorphic to $\mathbb{Z}_3 \times \mathbb{Z}_3$.
\item[\bf Subcase 1.] Assume that $m$ is odd.
Define $l : V(G) \to \mathbb{Z}_3 \times \mathbb{Z}_3 \setminus\{0\}$ by 
\[l(v_{i_k}) = \begin{cases} \text{$(1,0)$} &\quad\text{if $i > 1, i$ is even and $k \equiv 1(\textrm{mod}\ 4)$}\\ \text{$(1,2)$} &\quad\text{if $i > 1, i$ is even and $k \equiv 3(\textrm{mod}\ 4)$}\\  \text{$(2,0)$} &\quad\text{if $i > 1, i$ is odd and $k \equiv 1(\textrm{mod}\ 4)$}\\ \text{$(0,2)$} &\quad\text{if $i > 1, i$ is odd and $k \equiv 3(\textrm{mod}\ 4)$}\\ \text{$(1,1)$} &\quad\text{otherwise}\\ \end{cases} \] for all $i,k$ and $l(v)=(1,1)$
\item[\bf Subcase 2.]Assume that $m$ is even.
Define $l : V(G) \to \mathbb{Z}_3 \times \mathbb{Z}_3 \setminus\{0\}$ by 
\[l(v_{i_k}) = \begin{cases} \text{$(1,0)$} &\quad\text{if $i = 1$ and $k \equiv 1(\textrm{mod}\ 4)$}\\ \text{$(1,2)$} &\quad\text{if $i = 1$ and $k \equiv 3(\textrm{mod}\ 4)$}\\ \text{$(0,1)$} &\quad\text{if $i$ is even and $k \equiv 1(\textrm{mod}\ 4)$}\\  \text{$(2,1)$} &\quad\text{if $i$ is even and $k \equiv 3(\textrm{mod}\ 4)$}\\ \text{$(0,2)$} &\quad\text{if $i > 2, i$ is odd and $k \equiv 1(\textrm{mod}\ 4)$}\\ \text{$(2,0)$} &\quad\text{if $i > 2, i$ is odd and $k \equiv 3(\textrm{mod}\  4)$}\\ \text{$(1,1)$} &\quad\text{otherwise}\\ \end{cases} \] for all $i,k$ and $l(v)=(1,1)$. Thus $w(v)=(2,2)$ for all $v \in V(G).$\\
Suppose $A$ has a subgroup isomorphic to $\mathbb{Z}$, $\mathbb{Z}_9$ or $\mathbb{Z}_p$, where $p >3$, then the proof follows from Therorem $\ref{BP2002}$ and the labeling in case 3 of above theorem.
\end{proof}
 
\section{Concluding remarks}
In this paper we have proved the $A$-vertex magicness of several graphs by considering its reduced graph. These ideas helps us to embed every graph as an induced subgraph of an $A$-vertex magic graph and also to construct infinite families of $A$-vertex magic graphs in a much more simpler way.   

\section*{Acknowledgements} 
The work of S. Balamoorthy is supported by a Junior Research Fellowship from CSIR-UGC, India(UGC-Ref.No.:1085/(CSIR-UGC NET JUNE 2019)). The authors would like to thank Dr T. Kavaskar, Central University of Tamil Nadu, Thiruvarur for the support and fruitful discussions.

\end{document}